\documentclass[a4paper,dvips,11pt]{article}
\usepackage[latin1]{inputenc}
\usepackage{amsmath,amsthm,amsfonts}
\usepackage{graphicx,psfrag}
\usepackage{enumerate}
\usepackage[ps2pdf,colorlinks]{hyperref}
\usepackage[numbers]{natbib}

\newcommand{\R}{\mathbb{R}}
\newcommand{\Z}{\mathbb{Z}}

\newcommand{\cD}{\mathcal{D}}

\newcommand{\cF}{\mathcal{F}}

\newcommand{\cS}{\mathcal{S}}

\newcommand{\cU}{\mathcal{U}}

\newcommand{\hX}{\hat{X}}
\newcommand{\htau}{\hat{\tau}}
\newcommand{\tX}{\tilde{X}}
\newcommand{\tY}{\tilde{Y}}

\newcommand{\tpi}{\tilde{\pi}}
\newcommand{\barf}{\bar{f}}

\newcommand{\E}{\operatorname{E}}
\newcommand{\pr}{\operatorname{P}}
\newcommand{\ext}{\operatorname{ext}}

\newcommand{\st}{{\rm st}}

\newcommand{\eqst}{=_{\st}}
\newcommand{\lest}{\le_{\st}}

\newtheorem{lemma}{Lemma}
\newtheorem{proposition}{Proposition}

\newtheorem{theorem}{Theorem}

\newcommand{\mailto}[1]{\href{mailto:#1}{\nolinkurl{#1}}}
\newcommand{\reff}[1]{(\ref{#1})}
\numberwithin{equation}{section} \theoremstyle{plain}

\begin{document}

\title{Stability of parallel queueing systems \\ with coupled service rates\footnote{
This is a revised version (proof of Proposition~6 corrected) of the article originally published
in Discrete Event Dyn Syst (2008) 18:447--472,
\href{http://dx.doi.org/10.1007/s10626-007-0021-4}{doi:10.1007/s10626-007-0021-4}.}}

\author{
 Sem Borst\thanks{
 Eindhoven University of Technology,
 PO Box 513, 5600 MB Eindhoven, The Netherlands.
 URL: \url{http://ect.bell-labs.com/who/sem/}.
 Email: \protect\mailto{sem@win.tue.nl}.
 }
 \and
 Matthieu Jonckheere\thanks{
 Eindhoven University of Technology,
 PO Box 513, 5600 MB Eindhoven, The Netherlands.
 URL: \url{http://homepages.cwi.nl/\~jonckhee/}.
 Email: \protect\mailto{m.t.s.jonckheere@tue.nl}.
 }
 \and
 Lasse Leskel\"a\thanks{
 Aalto University,
 PO Box 11100, 00076 Aalto, Finland.
 URL: \url{http://www.iki.fi/lsl/}.
 Email: \protect\mailto{lasse.leskela@iki.fi}.
 }
}
\date{}

\maketitle

\begin{abstract}
This paper considers a parallel system of queues fed by independent arrival streams, where the
service rate of each queue depends on the number of customers in all of the queues. Necessary and
sufficient conditions for the stability of the system are derived, based on stochastic
monotonicity and marginal drift properties of multiclass birth and death processes. These
conditions yield a sharp characterization of stability for systems, where the service rate of each
queue is decreasing in the number of customers in other queues, and has uniform limits as the
queue lengths tend to infinity. The results are illustrated with applications where the stability
region may be nonconvex.
\end{abstract}

\noindent {\bf Keywords:} stability, positive recurrence, multiclass birth and death process,
Foster--Lyapunov drift criterion, stochastic comparison, coupled processors

\vspace{1ex}

\noindent {\bf AMS 2000 Subject Classification:} 60K25, 90B22, 68M20, 60K20

\vspace{1ex}

\section{Introduction}
\label{sec:intro}

We consider a parallel system of queues fed by independent arrival streams, where the service rate
of each queue depends on the number of customers in all of the queues. This type of model is
natural for manufacturing systems where a server is capable to process other queues when its own
buffer is empty, or for cellular radio networks, where the available transmission rate for
customers in a particular cell is decreasing in the number of customers in the neighboring
cells~\cite{bonald2004a}. Another important category of applications consists of processor sharing
models, where several customer classes simultaneously use one or more servers, whose rate
allocations and total processing rates may depend on the number of customers in each of the
classes~\cite{bonald2006}. For example, in wireless data networks employing channel-aware
scheduling, the total service rate available to all customers can be increasing in the total
number of customers, due to multiuser diversity~\cite{liu2003}.

Stability is arguably the most fundamental property of a queueing system, and provides a crude yet
useful first-order benchmark of the system performance. A general framework for analyzing
stochastic stability consists of Foster--Lyapunov criteria, which are based on finding a suitable
test function having a positive or negative mean drift in almost all states of the state
space~\cite{fayolle1995,meyn1993}. In the context of multiclass queueing systems with coupled
servers, these techniques have been successfully applied to systems with utility-based service
allocations~\cite{veciana2001}. Fluid limit analysis is another powerful method for finding
necessary and sufficient stability conditions for a wide class of multiclass queueing networks
with work-conserving service disciplines~\cite{dai1995,meyn1995}.

The stability analysis of multiclass queueing systems with general state-dependent service rates
is difficult, because there is no systematic way of finding test functions satisfying the
Foster--Lyapunov criteria, and the fluid limit techniques are often restricted to systems of
work-conserving servers with fixed total rate. An alternative means for deriving stability
conditions is to study whether the system of interest is stochastically comparable to a simpler
system that is easier to analyze. This approach was first used in the multiclass queueing context
by Rao and Ephremides~\cite{rao1988} and Szpankowski~\cite{szpankowski1988}, and later refined by
Szpankowski~\cite{szpankowski1994}, to characterize the stability of buffered random access
systems.

In this paper we provide an extension of the above ideas (tentatively discussed
in~\cite{jonckheere2006}), by deriving marginal drift criteria for multiclass birth and death
processes that allow us to analyze the stability of a broad class of parallel queueing systems.
Moreover, we present conditions for partial stability, where only some of the queues are stable,
and give a sharp stability characterization for systems, where the service rate of each queue is
decreasing in the number of customers in other queues, and has uniform limits as the queue lengths
tend to infinity. For systems of at most three queues, where the service rates only depend on
whether the queues are empty or not, our results yield as special cases stability
characterizations that have earlier been found using transform
methods~\cite{cohen1983,fayolle1979} and the ergodic theory of deflected random
walks~\cite{fayolle1995}.

The paper is organized as follows. Section~\ref{sec:model} describes the model details and
discusses a notion of stability that is convenient for the subsequent analysis.
Section~\ref{sec:mbd} presents a key coupling result and marginal drift criteria for the stability
of multiclass birth and death processes, while the main results regarding the stability of
queueing systems are given in Section~\ref{sec:queues}, in decreasing level of generality.
Section~\ref{sec:applications} illustrates the results with two applications, and
Section~\ref{sec:conclusion} concludes the paper.

\section{Model description}
\label{sec:model}

\subsection{Parallel queueing system with coupled service rates}
\label{sec:notation}

We consider a parallel system of $N$ queues, where each queue $i$ is fed by an independent Poisson
arrival process of rate $\lambda_i$, and served at rate $\phi_i(X_1,\dots,X_N)$ that depends on
the number of customers $X_j$ in each of the queues $j=1,\dots,N$. We assume that all customers
require independent exponentially distributed amounts of service with unit mean, and that the
system has unlimited buffer space to accommodate customers. The scheduling at each queue can be
first-come first-served, processor sharing, random order of service, or any discipline that does
not depend on the service requirements.

Under these assumptions we can model $X=(X_1,\dots,X_N)$ as a continuous-time Markov process on
$\Z_+^N$, with transitions $x \mapsto x+e_i$ occurring at rate~$\lambda_i$ and transitions
$x\mapsto x-e_i \ge 0$ at rate $\phi_i(x)$, where $e_i$ denotes the $i$-th unit vector in
$\Z_+^N$. We assume that the allocation function $\phi=(\phi_1,\dots,\phi_N)$ is bounded, which
guarantees that the process~$X$ is nonexplosive. Hence we may assume that $X$ and all other
stochastic processes treated in the sequel have paths in the space $D = D(\R_+,\Z_+^N)$ of
right-continuous functions from $\R_+$ to $\Z_+^N$ with finite left limits. Recall that a
stochastic process with paths in~$D$ can be viewed as a random element on the measurable space
$(D, \cD)$, where $\cD$ denotes the Borel $\sigma$-algebra generated by the standard Skorohod
topology~\cite{kallenberg2002}.

Observe that this model also covers scenarios where the service requirements of customers at
queue~$i$ are exponentially distributed with parameter $\mu_i \neq 1$, via replacing $\phi_i(x)$
by $\mu_i\phi_i(x)$.

\subsection{Stability notions}
\label{sec:stability}

A stochastic process~$X$ taking values in a countable space~$S$ and having paths in $D(\R_+,S)$ is
said to be \emph{stable}, if for any $\epsilon>0$ there exists a finite set~$K$ such that
\begin{equation}
 \label{eq:defStability}
 \pr(X(t) \notin K) \le \epsilon \quad \text{for all~$t$},
\end{equation}
and otherwise $X$ is said to be \emph{unstable}. Further, the process $X$ is called
\emph{transient}, if $X(t)\to \infty$ almost surely, that is, for any finite set~$K$,
\[
\pr(\bigcup_{s \ge 0} \bigcap_{t \ge s} \{X(t) \notin K\}) = 1.
\]
An alternative way to express~\reff{eq:defStability} is to say that the family of distributions
$\{\pr\circ X(t)^{-1}\}_{t \ge 0}$ is tight. Observe that an irreducible Markov process is stable
if and only if it is positive recurrent~\cite[Theorem 12.25]{kallenberg2002}. The following
proposition illustrates an intuitively clear relation between transience and instability.

\begin{proposition}
\label{the:transientUnstable}
Any transient stochastic process~$X$ having paths in $D(\R_+,S)$ is unstable.
\end{proposition}

\begin{proof}
If $X$ is transient, then for any finite set~$K$ there exists an~$s$ such that $\pr(\cap_{t \ge s}
\{X(t) \notin K\}) > 1/2$. Hence, $\sup_t \pr(X(t) \notin K) > 1/2$ for all finite $K$, so $X$
cannot be stable.
\end{proof}

In most applications it is natural to assume that the Markov process describing the system is
irreducible, in which case stability is equivalent to the existence of a unique stationary
distribution. In Section~\ref{sec:queues}, where the service rates of the original system are
modified in various ways, it may happen that some of the modified Markov processes are not
irreducible. This is why we need the following result to guarantee the existence of a stationary
distribution for a stable multiclass birth and death process under slightly weaker than usual
assumptions on the reachability of states. We denote by $X[x]$ the version of a Markov process $X$
started in state $x$.

\begin{proposition}
\label{the:stationaryDistribution}
Let $X$ be a $N$-class birth and death process with strictly positive birth rates~$\lambda_i$ and
bounded death rates $\phi_i(x)$. Then the following are equivalent:
\begin{enumerate}[(i)]
\item $X[x]$ is stable for some initial state~$x$.
\item $X[x]$ is stable for all initial states~$x$.
\item $X$ has a unique stationary distribution~$\pi$ supported on
a set~$C$ such that $\pr_x(X(t) \in \cdot) \to \pi$ in total variation for all initial states~$x$,
and $X[x]$ is irreducible and positive recurrent for any $x \in C$.
\end{enumerate}

Moreover, $X$ is unstable if and only if $X(t) \to \infty$ in probability, regardless of the
initial state.
\end{proposition}

\begin{proof}
See Appendix~\ref{sec:stationaryMeasure}.
\end{proof}

The following stability characterization of vector-valued stochastic processes is well-known.
Because the proof is short, we give it here for completeness.
\begin{proposition}
\label{the:stableMarginals}
A stochastic process $X = (X_1,\dots,X_N)$ taking values in a countable space $S_1 \times \dots
\times S_N$ is stable if and only if $X_i$ is stable for each~$i$.
\end{proposition}

\begin{proof}
First assume that $X$ is stable. Given $\epsilon>0$, let us fix a finite set~$K$ such that $\sup_t
\pr(X(t) \notin K) \le \epsilon$, and choose a finite rectangle $K_1 \times \dots \times K_N$ that
contains~$K$. Then for any~$i$ and all~$t$, $\pr(X_i(t) \notin K_i) \le \pr( X(t) \notin K)$, so
it follows that $X_i$ is stable.

For the reverse direction, it suffices to observe that for an arbitrary finite set $K = K_1 \times
\dots \times K_N$, $\pr(X(t) \notin K) \le \sum_i \pr(X_i \notin K_i)$.
\end{proof}

\section{Multiclass birth and death processes}
\label{sec:mbd}

\subsection{Stochastic comparison}
\label{sec:comparison}

When $X$ and $Y$ are random elements taking values in a partially ordered measurable space, we
denote $X \lest Y$ and say that $X$ is \emph{stochastically less} than $Y$, if $\E f(X) \le \E
f(Y)$ for all positive increasing measurable functions. We use the terms increasing and positive
in the weak sense, so that a function $f$ is increasing, if $f(x) \le f(y)$ for all $x \le y$, and
positive if $f(x)\ge 0$ for all $x$. Moreover, we denote $X \eqst Y$, if the distributions of $X$
and $Y$ are equal.

Let us endow the spaces $\Z_+^N$ and $D(\R_+,\Z_+^N)$ with the usual coordinate-wise partial
orders, so that $x \le y$ in $\Z_+^N$ if and only if $x_i \le y_i$ for all~$i$; and $x\le y$ in
$D(\R_+,\Z_+^N)$ if and only if $x_i(t) \le y_i(t)$ for all $i$ and $t$. Recall that by Strassen's
theorem~\cite[Theorem 1]{kamae1977}, the stochastic processes $X$ and $Y$ having paths in
$D(\R_+,\Z_+^N)$ satisfy $X \lest Y$ if and only if there exist processes $\tilde X$ and $\tilde
Y$ defined on a common probability space such that $\tilde X \eqst X$, $\tilde Y \eqst Y$, and
$\tilde X_i(t) \le \tilde Y_i(t)$ for all $i$ and $t$ almost surely. The following result
indicates a fundamental relation between stochastic ordering and stability.

\begin{proposition}
\label{the:comparisonStability}
Let $X$ and $Y$ be stochastic processes with paths in $D(\R_+,\Z_+^N)$ and assume that $X \lest
Y$.
\begin{enumerate}[(i)]
\item If $X$ is transient, then so is~$Y$.
\item If $Y$ is stable, then so is~$X$.
\end{enumerate}
\end{proposition}

\begin{proof}
The first claim is a direct consequence of Strassen's theorem, while the second follows directly
from the definition of stability, because $\pr(|X(t)| > r) \le \pr(|Y(t)| > r)$ for all~$r$
and~$t$.
\end{proof}

A Markov process having paths in $D(\R_+,\Z_+^N)$ is called a \emph{multiclass birth and death
process}, if its transitions are given by
\begin{align*}
 x &\mapsto x + e_i \quad \text{at rate $\lambda_i(x)$}, \\
 x &\mapsto x - e_i \quad \text{at rate $\phi_i(x) 1(x_i>0)$},
\end{align*}
where $\lambda_i(x)$ and $\phi_i(x)$ are some positive functions on $\Z_+^N$, called the class-$i$
birth rates and death rates, respectively. The following lemma, which is proved using a direct
coupling construction, gives a sufficient condition for the comparability of two multiclass birth
and death processes.

\begin{lemma}
\label{the:comparison}
Let $X=(X_1,\dots,X_I)$ and $Y=(Y_1,\dots,Y_J)$ be multiclass birth and death processes such that
$X$ has birth rates $\lambda_i(x)$ and death rates $\phi_i(x)$, and $Y$ has birth rates
$\eta_j(y)$ and death rates $\psi_j(y)$. Assume that for all $i = 1,\dots,I\wedge J$, and all $x
\in \Z_+^I$ and $y \in \Z_+^J$ such that $x_i = y_i$ and $(x_1,\dots,x_{I \wedge J}) \le
(y_1,\dots,y_{I \wedge J})$,
\begin{align}
\lambda_i(x) &\le \eta_i(y), \label{eq:comparisonUp} \\
\phi_i(x)    &\ge \psi_i(y). \label{eq:comparisonDown}
\end{align}
Then for all $x \in \Z_+^I$ and $y \in \Z_+^J$ such that $(x_1,\dots,x_{I \wedge J}) \le
(y_1,\dots,y_{I \wedge J})$,
\[
(X_1[x],\dots,X_{I \wedge J}[x]) \lest (Y_1[y],\dots,Y_{I \wedge J}[y]),
\]
where $X[x]$ and $Y[y]$ are versions of~$X$ and~$Y$ started in~$x$ and $y$, respectively.
\end{lemma}

\begin{proof}
Let $(\tX,\tY)$ be the Markov process with paths in $D(\R_+,U)$, where $U = \{(x,y)\in \Z_+^I
\times \Z_+^J: (x_1,\dots,x_{I \wedge J}) \le (y_1,\dots,y_{I \wedge J})\}$, having the upward
transitions
\begin{align*}
(x,y) &\mapsto (x+e_i,y) & \text{at rate} \ & \lambda_i(x),
\quad & i \le {I \wedge J}, \ x_i < y_i, \\
(x,y) &\mapsto (x,y+e_i) & \text{at rate} \ & \eta_i(y),
\quad & i \le {I \wedge J}, \ x_i<y_i, \\
(x,y) &\mapsto (x+e_i,y+e_i) & \text{at rate} \ & \lambda_i(x),
\quad & i \le {I \wedge J}, \ x_i=y_i, \\
(x,y) &\mapsto (x,y+e_i) & \text{at rate} \ & \eta_i(y) - \lambda_i(x),
\quad & i \le {I \wedge J}, \ x_i=y_i, \\
(x,y) &\mapsto (x+e_i,y) & \text{at rate} \ & \lambda_i(x),
\quad & i > {I \wedge J}, \\
(x,y) &\mapsto (x,y+e_i) & \text{at rate} \ & \eta_i(y), \quad & i > {I \wedge J},
\end{align*}
and the downward transitions
\begin{align*}
(x,y) &\mapsto (x-e_i,y) & \text{at rate} \ & \phi_i(x),
\quad & i \le {I \wedge J}, \ 0<x_i<y_i, \\
(x,y) &\mapsto (x,y-e_i) & \text{at rate} \ & \psi_i(y),
\quad & i \le {I \wedge J}, \ 0<x_i<y_i, \\
(x,y) &\mapsto (x-e_i,y-e_i) & \text{at rate} \ & \psi_i(y),
\quad & i \le {I \wedge J}, \ 0<x_i=y_i, \\
(x,y) &\mapsto (x-e_i,y) & \text{at rate} \ & \phi_i(x) - \psi_i(y),
\quad & i \le {I \wedge J}, \ 0<x_i=y_i, \\
(x,y) &\mapsto (x-e_i,y) & \text{at rate} \ & \phi_i(x),
\quad & i > {I \wedge J}, \\
(x,y) &\mapsto (x,y-e_i) & \text{at rate} \ & \psi_i(y), \quad & i > {I \wedge J}.
\end{align*}
In light of~\reff{eq:comparisonUp} and~\reff{eq:comparisonDown}, we see that all transition rates
described above are positive. Moreover, because each of the transitions are mappings from~$U$
into~$U$, we can be assured that the process $(\tX,\tY)$ exists.

By studying the marginals of the transition rates, we see that both $\tX$ and $\tY$ are Markov,
and that their intensity matrices coincide with those of~$X$ and~$Y$, respectively. Hence, for
all~$x$ and~$y$ such that $(x,y) \in U$, we have constructed versions of $X[x]$ and $Y[y]$ on a
common probability space such that $(X_1[x](t),\dots,X_{I \wedge J}[x](t)) \le
(Y_1[y](t),\dots,Y_{I \wedge J}[y](t))$ for all~$t$ almost surely.
\end{proof}

\subsection{Marginal drift conditions}
\label{sec:drift}

In this section we develop necessary and sufficient conditions for the stability of a multiclass
birth and death process, given that each
 coordinate process except one is known to be stable.
The following proposition extends the classical Neuts' mean drift condition~\cite{neuts1978}, see
also Tweedie~\cite{tweedie1982}. The proof follows closely the principles in Section~19 of Meyn
and Tweedie~\cite{meyn1993}.

\begin{proposition}
\label{the:marginalDrift}
Let $X = (X_1,\dots,X_{n+1})$ be an $(n+1)$-class birth and death process with strictly positive
birth rates~$\lambda_i$ and bounded death rates $\phi_i(x)$ such that $\phi_i(x) =
\phi_i(x_1,\dots,x_n)$ only depends on the first~$n$ input arguments for all $i = 1,\dots,n$.
Assume that:
\begin{enumerate}[(i)]
\item The Markov process $X^n=(X_1,\dots,X_n)$ is stable and has
the stationary distribution~$\pi$.
\item The birth rate of $X_{n+1}$ satisfies the condition
\[
\lambda_{n+1} < \sum_{x^n \in \Z_+^n} \left\{\liminf_{x_{n+1} \to \infty} \phi_{n+1}(x^n,x_{n+1})
\right\} \pi(x^n).
\]
\end{enumerate}
Then the process $X = (X^n,X_{n+1})$ is stable.

In particular, if $\phi_{n+1}(x) = \phi_{n+1}(x_{n+1})$ only depends on $x_{n+1}$, then
\[
\lambda_{n+1} < \liminf_{x_{n+1} \to \infty} \phi_{n+1}(x_{n+1})
\]
is sufficient for the stability of $X^{n+1}$, regardless of the initial state.
\end{proposition}

To prove the above proposition, we use the following lemma:

\begin{lemma}
\label{the:limsupBound}
Let $(X,Y)$ be a stochastic process with values in $S \times \Z_+$, where $S$ is countable. Assume
that
\begin{enumerate}[(i)]
\item $X_t \to \pi$ weakly for some probability distribution~$\pi$
on~$S$,
\item $Y_t \to \infty$ in probability.
\end{enumerate}
Then for all bounded $f$,
\begin{equation}
\label{eq:limsupBound}
\limsup_{t \to \infty} \E f(X_t,Y_t) \le \sum_x \left\{\limsup_{y \to \infty} f(x,y) \right\}
\pi(x).
\end{equation}
\end{lemma}

\begin{proof}
Assume first that $S$ is finite, and let $\barf(x) = \limsup_{y \to \infty} f(x,y)$. Then for any
$\epsilon>0$ there exists an~$r$ such that $f(x,y) \le \barf(x) + \epsilon$ for all $x$ and all $y
> r$. It follows that
\begin{align*}
 \E f(X_t,Y_t)
 &\le \E (\barf(X_t) + \epsilon) 1(Y_t > r) + \E f(X_t,Y_t) 1(Y_t \le r) \\
 &=   \E \barf(X_t) + \epsilon + \E [ f(X_t,Y_t) - \barf(X_t) - \epsilon ] 1(Y_t \le r) \\
 &\le \E \barf(X_t) + \epsilon + 2||f|| \pr(Y_t \le r),
\end{align*}
where $||f|| = \sup_{x,y} |f(x,y)|$. By letting $t \to \infty$ and recalling that $\epsilon$ is
arbitrary, we see that the claim holds for a finite set~$S$.

If $S$ is countably infinite, then for any finite set~$K$, the claim holds for the function
$f_K(x,y) = f(x,y) 1(x \in K)$. Hence, because $\barf_K(x) = \barf(x) 1(x \in K)$, we get
\[
\limsup_{t \to \infty} \E f(X_t,Y_t) \le \sum_x \barf(x) \, \pi(x) + 2 ||f|| \pi(K^c).
\]
Thus the claim follows, because we can make $\pi(K^c)$ arbitrarily small by choosing $K$ large
enough.
\end{proof}

\begin{proof}[Proof of Proposition~\ref{the:marginalDrift}]
Let us define $V(x) = x_{n+1}$ and denote the mean drift of~$V$ with respect to~$X$ by
\[
\Delta V(x) = \lambda_{n+1} - \phi_{n+1}(x)1(x_{n+1}>0).
\]
Let us also define $V_M(x) = V(x) 1(x_{n+1} \le M)$ for $M>0$. Then by Kolmogorov's forward
equation \cite[Theorem 19.6]{kallenberg2002} we have
\[
\int_0^t \E_x \Delta V_M(X(s)) \, ds = \E_x V_M(X(t)) - V_M(x).
\]
Because $V_M \to V$ and $\Delta V_M \to \Delta V$ pointwise as $M \to \infty$, and because
$||\Delta V_M|| \le \lambda + ||\phi||$, we see by applying dominated convergence on the left-hand
side, and monotone convergence on the right, that
\[
\int_0^t \E_x \Delta V(X(s)) \, ds = \E_x V(X(t)) - V(x).
\]
Because $V$ is positive, this implies that
\[
\limsup_{t \to \infty} \, t^{-1} \! \int_0^t \E_x \Delta V(X(s)) \, ds \ge 0,
\]
and consequently,
\[
\limsup_{t \to \infty} \E_x \Delta V(X(t)) \ge 0.
\]

Now assume that the process~$X$ is unstable. Then Proposition~\ref{the:stationaryDistribution}
implies that $X(t) \to \infty$ in probability. Because $X^n$ is stable, it follows that
$X_{n+1}(t) \to \infty$ in probability, and we may conclude by virtue of
Lemma~\ref{the:limsupBound} that
\[
\sum_{x^n \in \Z_+^n} \left\{\lambda - \liminf_{x_{n+1} \to \infty} \phi(x^n,x_{n+1}) \right\}
\pi(x^n) \ge \limsup_{t \to \infty} \E_x \Delta V(X(t)) \ge 0.
\]
\end{proof}

To prove the following converse of Proposition~\ref{the:marginalDrift}, we make the additional
assumption that $\phi_{n+1}$ only depends on the first $n$ input arguments. Hence $(X^n,X_{n+1})$
is a Markov additive process, where the state space of the modulating process $X^n$ is countably
infinite. For Markov additive processes where the modulating process only takes on finitely many
values, the following kind of result is well-known~\cite[Proposition~XI.2.11]{asmussen2003}.

\begin{proposition}
\label{the:marginalDrift2}
Let $X = (X_1,\dots,X_{n+1})$ be an $(n+1)$-class birth and death process with strictly positive
birth rates~$\lambda_i$ and bounded death rates $\phi_i(x)$ such that $\phi_i(x) =
\phi_i(x_1,\dots,x_n)$ only depends on the first~$n$ input arguments for all $i = 1,\dots,n+1$.
Assume that:
\begin{enumerate}[(i)]
\item The Markov process $X^n = (X_1,\dots,X_n)$ is stable and has
the stationary distribution~$\pi$.
\item The birth rate of $X_{n+1}$ satisfies the condition
\[
\lambda_{n+1} > \sum_{x^n \in \Z_+^n} \phi_{n+1}(x^n) \, \pi(x^n).
\]
\end{enumerate}
Then the process $X_{n+1}$ is transient, regardless of the initial state.
\end{proposition}

\begin{proof}
Let us define the free process $(X^n,X_{n+1}^f)$ as the Markov process with values in $\Z_+^n
\times \Z$, so that $X^n = (X_1,\dots,X_n)$ is as before, and conditional on~$X^n$, $X_{n+1}^f$ is
a birth and death process on~$\Z$ with constant birth rate $\lambda_{n+1}$ and time-varying
state-dependent death rates $\phi_{n+1}(X^n(t))$. The process $X_{n+1}^f$ can be represented as
\begin{equation}
  \label{eq:FreeProcess}
  X_{n+1}^f(t) = X_{n+1}^f(0) + N_b(\lambda_{n+1}t) - N_d\left(\int_0^t \phi_{n+1}(X^n(s)) \, ds\right),
\end{equation}
where $N_b$ and $N_d$ are independent unit-rate Poisson processes, independent of $X^n$. The
ergodic theorem for positive recurrent Markov processes \cite[Theorem 20.21]{kallenberg2002}
implies that
\[
 t^{-1} \int_0^t \phi_{n+1}(X^n(s)) \, ds \to \sum_{x^n\in \Z_+^n} \phi_{n+1}(x^n) \, \pi(x^n)
 \quad\text{a.s.}
\]
Therefore, dividing both sides of~\reff{eq:FreeProcess} by~$t$ and taking $t \to \infty$, we see
with the help of Lemma~\ref{the:LLN} (Appendix~\ref{sec:LLN}) that
\[
 \lim_{t \to \infty} t^{-1} X_{n+1}^f(t) = \lambda_{n+1} - \sum_{x^n\in \Z_+^n} \phi_{n+1}(x^n) \, \pi(x^n)
 \quad\text{a.s.}
\]
Because the limit is strictly positive, it follows that $X_{n+1}^f(t) \to \infty$ almost surely.

Finally, observe that whenever $(X^n,X_{n+1})$ and $(X^n,X_{n+1}^f)$ are started in the same
initial state, $X_{n+1}$ can be represented in terms of $X_{n+1}^f$ via the Skorohod map
\cite[Theorem D.1]{robert2003}
\[
X_{n+1}(t) = X_{n+1}^f(t) + \sup_{s \le t} [-X_{n+1}^f(s)]^+.
\]
Because $X_{n+1}^f(t) \to \infty$ almost surely, the same is true for $X_{n+1}$.
\end{proof}

\section{Stability results for queueing systems}
\label{sec:queues}

\subsection{General service allocations}
\label{sec:generalAllocations}

Let us return to the queueing system described in Section~\ref{sec:notation}, so from now on
$X=(X_1,\dots,X_N)$ describes the queue lengths of the system, and $\phi_i(x)$ is the service rate
for queue~$i$. The following result gives stability conditions valid for an arbitrary bounded
service allocation $\phi=(\phi_1,\dots,\phi_N)$. Although these conditions are not sharp in
general, they may provide useful inner and outer bounds for stability regions of complex systems
that are not easy to analyze exactly.

\begin{theorem}
\label{the:stable}
Let $X=(X_1,\dots,X_N)$ be the queue length process of the system with strictly positive arrival
rates~$\lambda_i$ and bounded service rates $\phi_i(x)$. Then, regardless of the initial state,
the process~$X_i$ is stable if
\begin{equation}
\label{eq:stableCond}
\lambda_i < \liminf_{x_i \to \infty} \, \inf_{x_j: j\neq i} \phi_i(x),
\end{equation}
and transient, if
\begin{equation}
\label{eq:unstableCond}
\lambda_i > \limsup_{x_i \to \infty} \, \sup_{x_j: j\neq i} \phi_i(x).
\end{equation}
In particular, $X$ is stable if~\reff{eq:stableCond} holds for all~$i$.
\end{theorem}

\begin{proof}
Assume that~\reff{eq:stableCond} holds for some~$i$. By relabeling the queues if necessary, we may
assume without loss of generality that $i=1$. Let $Y_1$ be the one-class birth and death process
with constant birth rate~$\lambda_1$ and death rates $\psi_1(x_1) = \inf_{x_2,\dots,x_N}
\phi_1(x)$. Because $\lambda_1 < \liminf_{x_1 \to \infty} \psi_1(x_1)$, it then follows from
Proposition~\ref{the:marginalDrift} that $Y_1$ is stable, regardless of the initial state.
Moreover, because $\phi_1(x) \ge \psi_1(x_1)$ for all~$x$, it follows from
Lemma~\ref{the:comparison} that $X_1[x] \lest Y_1[x_1]$ for any initial state~$x$. Hence, $X_1[x]$
is stable by Proposition~\ref{the:comparisonStability}.

Analogously, if~\reff{eq:unstableCond} holds, then it again suffices to consider $i=1$. In that
case, we let $Z_1$ be the one-class birth and death process with birth rate~$\lambda_1$ and death
rates $\chi_1(x_1) = \sup_{x_2, \dots, x_N} \phi_1(x) + \epsilon$, where $\epsilon>0$ is such that
$\lambda_1 - \epsilon$ is strictly larger than the right-hand side of~\reff{eq:unstableCond}. Then
$Z_1$ is irreducible and $\lambda_1 > \limsup_{x_1 \to \infty} \chi_1(x_1)$, so it follows from
the classical theory of ordinary birth and death processes \cite[Proposition
III.2.1]{asmussen2003} that $Z_1$ is transient. Applying Lemma~\ref{the:comparison} once more, we
see that $Z_1[x_1] \lest X_1[x]$ for any $x \in \Z_+^N$. Hence, $X_1[x]$ is transient by
Proposition~\ref{the:comparisonStability}.

Finally, if we assume that~\reff{eq:stableCond} holds for all~$i$, then all $X_i[x]$ are stable,
regardless of the initial state~$x$, hence $X$ is stable by Proposition~\ref{the:stableMarginals}.
\end{proof}

\subsection{Partially decreasing service allocations}
\label{sec:monotoneAllocations}

For the service allocation~$\phi$, the following notion of monotonicity is fundamental in
comparing multiclass birth and death processes. A function $\phi = (\phi_1,\dots,\phi_N)$ from
$\Z_+^N$ into $\R_+^N$ is said to be \emph{partially decreasing} if for all~$i$,
\begin{equation}
\label{monotone}
\phi_i(x) \ge \phi_i(y) \quad \text{for all $x\le y$ such that $x_i=y_i$}.
\end{equation}
Note that a function $\phi=(\phi_1,\dots,\phi_N)$ is partially decreasing, if each of the
coordinate functions~$\phi_i$ is decreasing. Moreover, any function $\phi = (\phi_1,\dots,\phi_N)$
such that $\phi_i$ only depends on~$x_i$, is partially decreasing. Recall also that a
continuous-time Markov process~$X$ is said to be \emph{monotone}, if the map $x \mapsto \E_x
f(X_t)$ is increasing for all~$t$ and for any bounded increasing function~$f$. Using a result of
Massey \cite[Theorem 5.2]{massey1987}, it can be checked that a multiclass birth and death
process~$X$ with constant birth rates and bounded state-dependent death rates~$\phi_i$ is monotone
if and only if $\phi$ is partially decreasing.

We  define $Y^n$ as the $n$-class birth and death process with birth rates~$\lambda_i$ and death
rates $\ell^n\phi_i$ given by the lower partial limits
\begin{equation}
\label{eq:lowerLimits}
\ell^n\phi_i(x_1,\dots,x_n) = \lim_{r \to \infty} \ \inf_{x_{n+1},\dots,x_N > r}
\phi_i(x_1,\dots,x_N).
\end{equation}
The process~$Y^n$ may intuitively be viewed as a partially saturated version of the queue length
process $X = (X_1,\dots,X_N)$, where queues $1,\dots,n$ are allocated the asymptotically
worst-case service rates as $X_{n+1},\dots,X_N$ tend to infinity. We also define
\begin{equation}
\label{eq:defL}
L^n_i(\lambda_1,\dots,\lambda_n; \phi) = \sum_{x \in \Z_+^n} \ell^n\phi_i(x) \, \pi^n(x),
\end{equation}
if $Y^n$ has a unique stationary distribution~$\pi^n$, and set $L^n_i(\lambda_1,\dots,\lambda_n;
\phi) = 0$ otherwise. Thus, the quantity $L^n_i(\lambda_1,\dots,\lambda_n; \phi)$ can be
interpreted as a worst-case average service rate dedicated to queue~$i$ in a partially saturated
system where the numbers of customers in queues $n+1,\dots,N$ tend to infinity. For notational
convenience, we define $L^n_i(\lambda_1,\dots,\lambda_n;\phi) = \ell^0 \phi_i$ for $n=0$.

\begin{theorem}
\label{the:stableDecr}
Let $X = (X_1,\dots,X_N)$ be the queue length process of the system with strictly positive arrival
rates~$\lambda_i$ and a bounded partially decreasing service allocation $\phi =
(\phi_1,\dots,\phi_N)$. Assume that there exists an~$n$ such that
\begin{equation}
\label{eq:nStability}
\lambda_i < L^{i-1}_i(\lambda_1,\dots,\lambda_{i-1};\phi)
\end{equation}
for all $i=1,\dots,n$. Then the processes $X_1,\dots,X_n$ are stable, regardless of the initial
state.
\end{theorem}

Before proving Theorem~\ref{the:stableDecr}, we establish some auxiliary results, the latter of
which will also be used in Section~\ref{sec:sharp}.

\begin{lemma}
\label{the:liminf}
Let $1 \le n \le N$. Then for all~$i$ and all~$x$,
\[
\ell^{n-1} \phi_i(x_1,\dots,x_{n-1}) \le \liminf_{x_n \to \infty} \, \ell^n\phi_i(x_1,\dots,x_n).
\]
\end{lemma}

\begin{proof}
Fix an $x \in \Z_+^n$ and denote $\alpha = \ell^{n-1} \phi_i(x_1,\dots,x_{n-1})$. Then for any
$\epsilon>0$ there exists an~$r$ such that
\[
\alpha - \epsilon \le \inf_{y_n,\dots y_N \ge r} \phi_i(x_1,\dots,x_{n-1},y_n,\dots,y_N).
\]
Hence, it follows that for all $y_n \ge r$ and for all $s \ge r$,
\begin{align*}
\alpha - \epsilon
&\le \inf_{y_{n+1}, \dots y_N \ge r} \phi_i(x_1,\dots,x_{n-1},y_n,\dots,y_N) \\
&\le \inf_{y_{n+1}, \dots y_N \ge s} \phi_i(x_1,\dots,x_{n-1},y_n,\dots,y_N).
\end{align*}
By taking $s \to \infty$, it follows that $\alpha - \epsilon \le \ell^n
\phi_i(x_1,\dots,x_{n-1},y_n)$ for all $y_n \ge r$, so the claim follows because $\epsilon$ is
arbitrary.

\end{proof}

\begin{lemma}
\label{the:stableY}
Assume that the function $\phi = (\phi_1,\dots,\phi_N)$ is bounded and partially decreasing, and
assume that $\lambda_i$ are strictly positive and satisfy inequalities~\reff{eq:nStability} for
$i=1,\dots,n$. Then the $n$-class birth and death process $Y^n = (Y^n_1,\dots,Y^n_n)$ with birth
rates~$\lambda_i$ and death rates~$\ell^n\phi_i$ is stable.
\end{lemma}

\begin{proof}
To prove the claim for $n=1$, let us assume that $\lambda_1 < \ell^0 \phi_1$. Then $Y^1$ is a
one-class birth and death process with birth rate~$\lambda_1$ and state-dependent death rates
$\ell^1 \phi_1(x_1)$. The stability of~$Y_1$ follows from Proposition~\ref{the:marginalDrift},
because by Lemma~\ref{the:liminf},
\[
\lambda_1 < \ell^0 \phi_1 \le \liminf_{x_1\to\infty} \ell^1 \phi_1(x_1).
\]

To proceed by induction, suppose that the claim is true for $n-1$, and assume that the
inequalities~\reff{eq:nStability} hold for $i = 1,\dots,n$. Then $Y^{n-1}$ is stable by the
induction assumption, and Proposition~\ref{the:stationaryDistribution} shows that $Y^{n-1}$ has a
unique stationary distribution $\pi^{n-1}$. Let $Z^n = (Z_1^n,\dots,Z_n^n)$ be the $n$-class birth
and death process with birth rates~$\lambda_i$ and death rates
\[
\psi_i(x_1,\dots,x_n) = \left\{
\begin{aligned}
\ell^{n-1}\phi_i(x_1, \dots, x_{n-1}), \quad i&<n, \\
\ell^{n}\phi_i(x_1,\dots,x_n),         \quad i&=n.
\end{aligned}
\right.
\]
This choice of rates implies that the marginal process $(Z_1^n,\dots,Z_{n-1}^n)$ is Markov and
coincides with $Y^{n-1}$. Now observe that the condition $\lambda_n <
L^{n-1}_n(\lambda_1,\dots,\lambda_{n-1};\phi)$ is equivalent to
\[
\lambda_n < \sum_{x\in\Z_+^{n-1}} \ell^{n-1} \phi_n(x_1,\dots,x_{n-1}) \,
\pi^{n-1}(x_1,\dots,x_{n-1}),
\]
so using Lemma~\ref{the:liminf}, we see that
\[
\lambda_n < \sum_{x \in \Z_+^{n-1}} \left\{\liminf_{x_n \to \infty} \psi_n(x_1,\dots,x_{n-1},x_n)
\right\} \pi^{n-1}(x_1,\dots,x_{n-1}).
\]
Hence, Proposition~\ref{the:marginalDrift} shows that $Z^n$ is stable.

Because $\phi$ is partially decreasing, we have $\ell^{n-1} \phi_i(x_1,\dots,x_{n-1}) \le
\phi_i(x)$ for all $i<n$, and thus $\ell^{n-1} \phi_i(x_1,\dots,x_{n-1}) \le
\ell^n\phi_i(x_1,\dots,x_n)$ for all $i<n$. In particular, it follows that $\psi_i(x) \le
\ell^n\phi_i(x)$ for all~$x$ and all $i\le n$. Further, because $\ell^n\phi$ is also partially
decreasing, it follows that $\ell^n\phi_i(x) \ge \psi_i(y)$ for all $x$ and $y$ in $\Z_+^n$ such
that $x\le y$ and $x_i=y_i$. Thus, Lemma~\ref{the:comparison} implies that $Y^n \lest Z^n$,
whenever $Y^n$ and $Z^n$ are started in the same initial state. Hence, $Y^n$ is stable by
Proposition~\ref{the:comparisonStability}.
\end{proof}

\begin{proof}[Proof of Theorem~\ref{the:stableDecr}]
Assuming the inequalities~\reff{eq:nStability} are valid for $i = 1,\dots, n$, we see using
Lemma~\ref{the:stableY} that $Y^n$ is stable. Moreover, $\phi_i(x_1,\dots,x_N) \ge \ell^n
\phi_i(x_1,\dots, x_n)$ for all $i \le n$ and all $x \in \Z_+^N$, because $\phi$ is partially
decreasing. Because $\ell^n \phi$ is also partially decreasing, it follows that $\phi_i(x) \ge
\ell^n \phi_i(y)$ for all $x \in \Z_+^N$ and $y \in \Z_+^n$ such that $(x_1,\dots,x_n) \le
(y_1,\dots,y_n)$ and $x_i=y_i$. Hence, Lemma~\ref{the:comparison} shows that
$(X_1[x],\dots,X_n[x]) \lest (Y_1^n[y],\dots,Y^n_n[y])$ for all initial states $x \in \Z_+^N$ and
$y \in \Z_+^n$ such that $(x_1,\dots,x_n) \le (y_1,\dots,y_n)$.
Proposition~\ref{the:comparisonStability} thus implies that $(X_1[x],\dots,X_n[x])$ is stable.
\end{proof}

\subsection{Partially decreasing service allocations with uniform limits}
\label{sec:uniformAllocations}

In the following we restrict ourselves to queueing systems where the service allocation $\phi =
(\phi_1,\dots,\phi_N)$ is such that each coordinate function~$\phi_i$ has a uniform limit as some
of the input variables tend to infinity. More precisely, we say that a function $f: \Z_+^N \to \R$
has \emph{uniform limits at infinity}, if the following conditions hold:
\begin{enumerate}[(i)]
\item There exists a constant~$f^0$ such that
\[
 \sup_{x \in \Z_+^N: x_1,\dots,x_N > r} | f(x) - f^0 |  \to 0,
\quad \text{as $r \to \infty$}.
\]
\item For any $n = 1,\dots,N-1$ and any permutation~$\sigma$ on
$\{1,\dots,N\}$, there exists a function $f^{n,\sigma}: \Z_+^n \to \R$ such that as $r \to
\infty$,
\[
\sup_{x \in \Z_+^N: x_{\sigma(n+1)},\dots,x_{\sigma(N)} > r} | f(x) -
f^{n,\sigma}(x_{\sigma(1)},\dots,x_{\sigma(n)}) | \to 0.
\]
\end{enumerate}

We will next show that the class of functions having uniform limits at infinity is rich enough to
be of interest. For example, assume that the allocation function~$\phi$ is of the form
\[
\phi_i(x) = g_i(x_i) h(x)
\]
where $g_i$ has a limit at infinity and $h$ is decreasing. Then $g_i$ obviously has uniform limits
at infinity, and hence the same applies for~$\phi_i$ using the result below.

\begin{proposition}
\label{the:uniformLimits}
Let $f$ and $g$ be bounded functions on $\Z_+^N$.
\begin{enumerate}[(i)]
\item If $f$ is positive and decreasing, then it has uniform
limits at infinity.
\item If $f$ and $g$ have uniform limits at infinity,
then so do the functions $f+g$ and $fg$.
\end{enumerate}
\end{proposition}

\begin{proof}
See Appendix~\ref{sec:uniformLimits}.
\end{proof}

If $\phi$ has uniform limits, then the partial lower limits~$\ell^n \phi$ defined
in~\reff{eq:lowerLimits} become true limits in the sense that
\[
\phi(x_1,\dots,x_N) \to \ell^n \phi(x_1,\dots,x_n)
\]
uniformly over $x_1,\dots,x_n$, as $\min(x_{n+1},\dots,x_N) \to \infty$.

\begin{theorem}
\label{the:unstableDecr}
Let $X = (X_1,\dots,X_N)$ be the queue length process of the system with strictly positive arrival
rates~$\lambda_i$ and a partially decreasing service allocation $\phi = (\phi_1,\dots,\phi_N)$
having uniform limits at infinity. Assume that there is an index~$n$ such that
\begin{align}
\label{eq:unstableDecrCond}
\lambda_i &< L^{i-1}_i(\lambda_1,\dots,\lambda_{i-1};\phi)
\quad &\text{for all} \ i \le n, \\
\label{eq:unstableDecrCond2}
\lambda_i &> L^n_i(\lambda_1,\dots,\lambda_n; \phi) \quad &\text{for all} \ i>n.
\end{align}
Then the process $(X_{n+1},\dots,X_N)$ is unstable, regardless of the initial state.
\end{theorem}

\begin{proof}
Given $\epsilon \ge 0$, let $Y^{n,\epsilon}$ be the $n$-class birth and death process with birth
rates~$\lambda_i$ and death rates $\ell^n\phi_i(x_1,\dots,x_n) + \epsilon$, and let $Y^n =
Y^{n,0}$. Then $Y^n$ is stable by Lemma~\ref{the:stableY}. Further, because $\ell^n\phi$ is
partially decreasing, it follows by Lemma~\ref{the:comparison} that $Y^{n,\epsilon} \lest Y^n$,
whenever $Y^{n,\epsilon}$ and $Y^n$ are started in the same state. Hence,
Lemma~\ref{the:comparisonStability} implies that $Y^{n,\epsilon}$ is stable for all $\epsilon>0$,
and moreover $\pi^{n,\epsilon} \lest \pi^n$, where $\pi^{n,\epsilon}$ and $\pi^n$ denote the
stationary distributions of $Y^{n, \epsilon}$ and~$Y^n$, respectively. In particular, the family
$\{\pi^{n, \epsilon}\}_{\epsilon > 0}$ is tight, and so by Lemma~\ref{the:continuity}
(Appendix~\ref{sec:perturbation}), $\pi^{n,\epsilon} \to \pi^n$ weakly, as $\epsilon \to 0$.
Because $\phi$ is bounded, it follows that for all~$i$,
\[
\lim_{\epsilon \to 0} \sum_{x \in \Z_+^n} \ell^n\phi_i(x) \, \pi^{n,\epsilon}(x) =
L^n_i(\lambda_1,\dots,\lambda_n; \phi).
 \]
Hence, by~\reff{eq:unstableDecrCond2}, we can choose an $\epsilon>0$ such that
\begin{equation}
\label{eq:notcondn}
\lambda_i > \sum_{x\in\Z_+^n} \ell^n\phi_i(x) \, \pi^{n,\epsilon}(x) \quad \text{for all $i>n$}.
\end{equation}

Let $Z$ be the $N$-class birth and death process with birth rates~$\lambda_i$ and death rates
$\psi_i(x_1,\dots,x_N) = \ell^n \phi_i(x_1,\dots,x_n) + \epsilon$. Then $(Z_1,\dots,Z_{n})$ is
Markov and coincides with $Y^{n,\epsilon}$ as defined above. Moreover,
inequality~\reff{eq:notcondn} implies that for all $i > n$, the mean drift of~$Z_i$ is strictly
positive, so Proposition~\ref{the:marginalDrift2} implies that $Z_i$ is transient for each $i>n$.

Next, observe that because $\phi_i$ has uniform limits, it follows that there exists an~$r_0$ such
that $\psi_i(x) \ge \phi_i(x)$ for all $i = 1,\dots,N$ and for all~$x$ such that
$x_{n+1},\dots,x_N \ge r_0$. Let us choose an $r \ge r_0$ and define $A_r$ to be the complement of
the set
\[
B_r = \{x \in \Z_+^N: x_{n+1},\dots,x_N \ge r\},
\]
Because each of the processes $Z_{n+1},\dots,Z_N$ is transient, it follows that we can choose an
$x \in B_r$ such that
\begin{equation}
\label{eq:transientZ}
\pr_x(\tau_{A_r}(Z) = \infty) > 0,
\end{equation}
where $\tau_{A_r}(Z) = \inf\{t>0: Z(t)\in A_r\}$ denotes the hitting time of~$Z$ into~$A_r$.
Because $\phi$ is partially decreasing, it follows that $\psi_i(x) \ge \phi_i(y)$ for all $x$ and
$y$ in $B_r$ such that $x\le y$ and $x_i = y_i$. Using a similar coupling construction as in
Lemma~\ref{the:comparison}, it is then straightforward to verify that $\tau_{A_r}(Z[x]) \lest
\tau_{A_r}(X[x])$ for all $x \in B_r$. Hence using~\reff{eq:transientZ} we may conclude that there
exists an $x\in B_r$ such that
\begin{equation}
\label{eq:noHit}
\pr_x(\tau_{A_r}(X) = \infty) > 0.
\end{equation}

We now assume that $X$ is stable and derive a contradiction. Using
Proposition~\ref{the:stationaryDistribution}, we know that there exists a set~$C$ such that $X[x]$
is irreducible and positive recurrent for all $x \in C$. Let us now choose an $r \ge r_0$ such
that $C \cap A_r$ is nonempty. Using Lemma~\ref{the:increasingSets} and standard properties of
irreducible Markov processes \cite[Proposition 8.13]{kallenberg2002}, we then know that the
hitting time of $X[x]$ into $C \cap A_r$ is finite almost surely for all~$x$, which
contradicts~\reff{eq:noHit}. Hence, $X$ must be unstable. In particular, because $(X_1,\dots,X_n)$
is stable, it follows from Proposition~\ref{the:stableMarginals} that $(X_{n+1},\dots,X_N)$ is
unstable.
\end{proof}

\label{sec:sharp}
We are now ready to present our main theorem. For any permutation~$\sigma$ on $\{1,\dots,N\}$, we
define $\lambda^\sigma_i = \lambda_{\sigma(i)}$ and $\phi^\sigma_i(x) =
\phi_{\sigma(i)}(x_{\sigma^{-1}(1)},\dots,x_{\sigma^{-1}(N)})$. The vector~$\lambda^\sigma$ and
the function~$\phi^\sigma$ will then correspond to the system where queues are relabeled according
to~$\sigma$. Denote
\[
S_N(\phi) = \left\{\lambda\in (0,\infty)^N: \lambda_i < L_i^{i-1}(\lambda_1,\dots,\lambda_{i-1};
\phi) \ \forall i=1,\dots,N\right\},
\]
where $L_i^{i-1}(\lambda_1,\dots,\lambda_{i-1}; \phi)$ are given by~\reff{eq:defL}, and define
\begin{equation}
\label{eq:defS}
\cS(\phi) = \bigcup_\sigma \left\{\lambda \in (0,\infty)^N: \lambda^\sigma \in
S_N(\phi^\sigma)\right\},
\end{equation}
where the union is taken over all permutations on $\{1,\dots,N\}$.

\begin{theorem}
\label{the:boundaryDecr}
Assume $\phi = (\phi_1,\dots,\phi_N)$ is partially decreasing and has uniform limits at infinity.
Then the set $\cS(\phi)$ defined by~\reff{eq:defS} is open, and the queue length process
$X=(X_1,\dots,X_N)$ of the system with arrival rates $\lambda = (\lambda_1,\dots,\lambda_N)$ and
service allocation~$\phi$ is stable for all $\lambda \in \cS(\phi)$, and unstable for
all~$\lambda$ outside the closure of $\cS(\phi)$.
\end{theorem}

Note that it is not possible to characterize the stability of the system for arrival rate vectors
belonging to the boundary of $\cS(\phi)$ without more detailed information on the allocation
function~$\phi$. Consider for example the one-server system where $\phi_1(x_1) =
(1+1/{x_1})^\alpha$ for some $\alpha \ge 0$. Then $\cS(\phi) = \{\lambda_1: \lambda_1<1\}$, and
the system with $\lambda_1 = 1$ is positive recurrent if $\alpha>1$ and null recurrent otherwise
\cite[Section III.2]{asmussen2003}.

\begin{lemma}
\label{the:phiContinuity}
Let $\phi = (\phi_1,\dots,\phi_N)$ be bounded and partially decreasing. Then for all
$n=1,\dots,N$, the set
\[
S_n(\phi) = \left\{\lambda\in (0,\infty)^N: \lambda_i < L_i^{i-1}(\lambda_1,\dots,\lambda_{i-1};
\phi) \ \forall i=1,\dots,n \right\}
\]
is open, and the function $f^n = (f^n_1,\dots,f^n_n)$ on $(0,\infty)^N$ defined by
\begin{equation}
\label{eq:etaFunction}
f_i^n(\lambda_1,\dots,\lambda_N) = L_i^{i-1}(\lambda_1,\dots,\lambda_{i-1}; \phi)
\end{equation}
is continuous on $S_{n-1}(\phi)$.
\end{lemma}

\begin{proof}
First observe that the function $f^1$, being a constant, is continuous on the open set $S_0(\phi)
= (0,\infty)^N$. To proceed by induction, let us next assume that $S_{n-1}(\phi)$ is open and
$f^n$ is continuous in $S_{n-1}(\phi)$ for some~$n$. To show that $S_n(\phi)$ is open, assume that
$S_n(\phi)$ is nonempty, and choose a vector $\lambda \in S_n(\phi)$. Then
\begin{equation}
\label{eq:lambdacont}
\lambda_{n} < L_n^{n-1}(\lambda_1,\dots,\lambda_{n-1}; \phi).
\end{equation}
Moreover, because $S_n(\phi) \subseteq S_{n-1}(\phi)$, it follows that $f^n$ is continuous
at~$\lambda$, so in particular the map $\eta \mapsto L_n^{n-1}(\eta_1,\dots,\eta_{n-1}; \phi)$ is
continuous at~$\lambda$. This implies that~\reff{eq:lambdacont} remains valid in some open
neighborhood $B_\lambda$ of~$\lambda$. Further, because $S_{n-1}(\phi)$ is open, there is another
open neighborhood~$B_\lambda'$ of~$\lambda$ such that $B_{\lambda}' \subseteq S_{n-1}(\phi)$. It
follows that $B_\lambda \cap B_\lambda' \subseteq S_n(\phi)$, so we may conclude that $S_n(\phi)$
is open.

To complete the induction, we next prove the continuity of~$f^{n+1}$ on $S_n(\phi)$, under the
assumptions that $S_n(\phi)$ is open and $f^n$ is continuous on~$S_{n-1}(\phi)$. If $\lambda \in
S_{n}(\phi)$, then $f^n$ is continuous at~$\lambda$, because $S_n(\phi) \subseteq S_{n-1}(\phi)$.
Because the first~$n$ coordinate functions of~$f^{n+1}$ coincide with~$f^n$, it is sufficient to
prove that the function $\eta \mapsto L_{n+1}^n(\eta_1,\dots,\eta_n; \phi)$ is continuous
at~$\lambda$. Thus, let $\lambda^k$ be a sequence converging to~$\lambda$. Because $S_n(\phi)$ is
open, we may assume without loss of generality that there exists a vector $\lambda' \in S_n(\phi)$
such that $\lambda^k \le \lambda'$ for all~$k$. Let $Y$ be the $n$-class birth and death process
with birth rates~$\lambda_i$ and death rates $\ell^n \phi_i(x)$, $i = 1,\dots,n$, and let $Y^k$
and $Y'$ be the corresponding processes with $\lambda$ replaced by~$\lambda^k$ and~$\lambda'$,
respectively. Then by Lemma~\ref{the:stableY}, all of the processes~$Y$, $Y^k$, and~$Y'$ are
stable, so we denote their stationary distributions by~$\pi$, $\pi^k$, and~$\pi'$, respectively.
Moreover, because $\ell^n\phi$ is partially decreasing, Lemma~\ref{the:comparison} shows that
$\pi^k \lest \pi'$ for all~$k$, so the family $\{\pi^k\}_{k \ge 0}$ is tight. Hence, we can apply
Lemma~\ref{the:continuity} (Appendix~\ref{sec:perturbation}) to see that $\pi^k \to \pi$ weakly,
so it follows from the boundedness of~$\phi$ that $L_{n+1}^n(\lambda_1^k,\dots,\lambda_n^k; \phi)
\to L_{n+1}^{n}(\lambda_1,\dots,\lambda_{n}; \phi)$. This shows that $f^{n+1}$ is continuous
at~$\lambda$.
\end{proof}

\begin{proof}[Proof of Theorem~\ref{the:boundaryDecr}]
First, let us note that by using Lemma~\ref{the:phiContinuity} and relabeling the classes if
necessary, we see that the set $\{\lambda: \lambda^\sigma \in S_N(\phi^\sigma)\}$ is open for
all~$\sigma$. Hence, the set $\cS(\phi)$ is open. Moreover, by again relabeling the classes if
necessary, we find using Theorem~\ref{the:stableDecr} that $\lambda \in \cS(\phi)$ implies the
stability of the queueing system.

To study the instability of the system, let us define analogously to~\reff{eq:defS} the set
\[
\cU(\phi) = \bigcup_\sigma \left\{ \lambda\in (0,\infty)^N: \lambda^\sigma \in U_N(\phi^\sigma)
\right\},
\]
where the union is taken over all permutations on $\{1,\dots,N\}$, and $U_N(\phi)$ is defined as
the set of $\lambda \in (0,\infty)^N$ such that the inequalities~\reff{eq:unstableDecrCond}
and~\reff{eq:unstableDecrCond2} are valid for some $n \in \{0,\dots,N-1\}$. Then by first
relabeling the classes if necessary, and then using Theorem~\ref{the:unstableDecr}, we see that
the system is instable for all $\lambda \in \cU(\phi)$.

It remains to show the instability of the system under the assumption that $\lambda$ belongs to
the complement of the closure of~$\cS(\phi)$ in $(0,\infty)^N$, which we denote by $\ext(\cS)$. We
prove this by showing that for each $\lambda \in \ext(\cS)$ there exists a $\tilde\lambda \in
\cU(\phi)$ such that $\tilde\lambda \le \lambda$ and then applying
Proposition~\ref{the:comparisonStability}.

Let $\lambda \in \ext(\cS)$. Given a permutation~$\sigma$, let us define $n(\lambda,\sigma)$ as
the largest integer~$n$ such that
\begin{equation}
\label{eq:condi}
\lambda^\sigma_i < L_i^{i-1}(\lambda^\sigma_1,\dots,\lambda^\sigma_{i-1}; \phi^\sigma) \quad
\text{for all $i<n$}.
\end{equation}
Because $\lambda \notin \cS(\phi)$, we know that $n(\lambda,\sigma) < N$ for each~$\sigma$, and
\begin{equation}
\label{eq:notCondn}
\lambda^\sigma_i \ge L_i^{i-1}(\lambda^\sigma_1,\dots,\lambda^\sigma_{i-1}; \phi^\sigma) \quad
\text{for $i=n(\lambda,\sigma)$}.
\end{equation}

Let $D$ be the set of~$\sigma$ for which~\reff{eq:notCondn} is strict, and assume $\sigma \in D$.
Then by Lemma~\ref{the:phiContinuity} we see that the function
\begin{equation}
\label{eq:function}
(\lambda^\sigma_1,\dots,\lambda^\sigma_{n-1}) \mapsto \left(L_1^{0}(\phi^\sigma), \dots,
L_n^{n-1}(\lambda_1^\sigma,\dots,\lambda^\sigma_{n-1};\phi^\sigma) \right)
\end{equation}
is continuous in a neighborhood of $(\lambda^\sigma_1,\dots,\lambda^\sigma_{n-1})$. Thus we can
choose an $\epsilon_\sigma \in (0, \min_j \lambda_j)$ such that the inequalities~\reff{eq:condi}
remain valid and the inequality~\reff{eq:notCondn} remains strict when $\lambda$ is replaced
by~$\tilde\lambda$ such that $|| \tilde\lambda - \lambda || < \epsilon_\sigma$.

On the other hand, if $\sigma \notin D$, then again by the continuity of the function
in~\reff{eq:function}, there exists an $\epsilon_\sigma \in (0, \min_j \lambda_j)$ such that the
inequalities~\reff{eq:condi} remain valid when $\lambda$ is replaced by $\tilde\lambda$ such that
$|| \tilde\lambda - \lambda || < \epsilon_\sigma$. Moreover, if we let $\tilde\lambda = \lambda -
r e$ for some $r \in (0,\epsilon_\sigma)$, then it can be checked using Lemma~\ref{the:comparison}
that $L_n^{n-1}(\lambda^\sigma_1,\dots,\lambda^\sigma_{n-1}; \phi^\sigma) \le L_n^{n-1}(\tilde
\lambda^\sigma_1,\dots, \tilde \lambda^\sigma_{n-1}; \phi^\sigma)$, because the function
$\ell^{n-1}\phi^\sigma$ is partially decreasing, and in particular, $\ell^{n-1}\phi_n^\sigma$ is
decreasing. Thus we see that for $n = n(\lambda,\sigma)$,
\begin{align*}
\tilde\lambda_n^\sigma < \lambda_n^\sigma
&= L_n^{n-1}(\lambda^\sigma_1,\dots,\lambda^\sigma_{n-1}; \phi^\sigma) \\
&\le L_n^{n-1}(\tilde \lambda^\sigma_1,\dots, \tilde \lambda^\sigma_{n-1}; \phi^\sigma).
\end{align*}

Let $\epsilon = \min_{\sigma} \epsilon_\sigma$ and choose a small enough $r \in (0,\epsilon)$ such
that $\tilde\lambda = \lambda - r e$ belongs to $\ext(\cS)$. Then it follows from the above
observations that for all~$\sigma$,
\begin{align*}
 \tilde\lambda^\sigma_i &<
 L_i^{i-1}(\tilde\lambda_1^\sigma,\dots,\tilde\lambda^\sigma_{i-1}; \phi^\sigma),
 &\quad i<n(\lambda,\sigma), \\
 \tilde \lambda_i^\sigma &\neq
 L_i^{i-1}(\tilde\lambda_1^\sigma,\dots,\tilde\lambda^\sigma_{i-1}; \phi^\sigma),
 &\quad i=n(\lambda,\sigma).
\end{align*}
Hence, either $\tilde\lambda_n^\sigma > L_n^{n-1}(\tilde\lambda_1^\sigma, \dots,
\tilde\lambda^\sigma_{n-1}; \phi^\sigma)$ for all~$\sigma$ and all $n = n(\lambda,\sigma)$, which
implies that $\tilde\lambda \in \cU(\phi)$; or else,
\begin{equation}
\label{eq:nineq}
n(\tilde\lambda,\sigma) \ge n(\lambda,\sigma) \quad\text{for all $\sigma$},
\end{equation}
and $\tilde \lambda_n^\sigma < L_n^{n-1}(\tilde\lambda_1^\sigma, \dots,
\tilde\lambda^\sigma_{n-1}; \phi^\sigma)$ for at least some~$\sigma$ and $n=n(\lambda,\sigma)$,
which implies that~\reff{eq:nineq} holds strictly for at least one~$\sigma$.

Assuming $\lambda^k \in \ext(\cS) \setminus \cU(\phi)$, we can apply the above procedure to
$\lambda^k$ to find a $\lambda^{k+1} \in \ext(\cS)$ such that $\lambda^{k+1} \le \lambda^k$, and
so that either $\lambda^{k+1} \in \cU(\phi)$, or else
\begin{equation}
\label{eq:nineq2}
n(\lambda^{k+1},\sigma) \ge n(\lambda^k,\sigma) \quad\text{for all $\sigma$},
\end{equation}
where~\reff{eq:nineq2} holds strictly for at least one~$\sigma$. The sequence $\lambda^k$ started
in $\lambda^1 = \lambda$ must hit $\cU(\phi)$ for some~$k$, because otherwise $n(\lambda^k,\sigma)
= N+1$ eventually for some~$k$ and some~$\sigma$, which would imply that $\lambda^k \in
\cS(\phi)$. Hence, $\lambda^k \in \cU$ for some~$k$, and $\lambda^k \le \lambda$. Now by
Proposition~\ref{the:comparisonStability}, it follows that the system is unstable.
\end{proof}

\section{Applications}
\label{sec:applications}

\subsection{Three weakly coupled queues}
\label{sec:threeProcessors}

Consider a system of three queues where the service rates at each queue only depend on whether the
other queues are empty or not, so that for all $i\neq j\neq k$,
\[
 \phi_i(x) = \left\{
 \begin{aligned}
  a_i   , \quad & x_j=0, \ x_k=0,\\
  a_{ij}, \quad & x_j>0, \ x_k=0,\\
  1,      \quad & x_j>0, \ x_k>0.
 \end{aligned}
 \right.
\]
Let us assume $a_i \ge a_{ij} \ge 1$, so that $\phi=(\phi_1,\phi_2,\phi_3)$ is partially
decreasing.

Theorem~\ref{the:boundaryDecr} shows that the stability region is a union of six regions
corresponding to the six possible permutations of the queues. The first of these regions
corresponding to the identity permutation is the set of $(\lambda_1,\lambda_2,\lambda_3)$ such
that
\begin{align}
\lambda_1 &< 1, \label{c31} \\
\lambda_2 &< \lambda_1 + a_{23} (1 -\lambda_1) \label{c32}, \\
\lambda_3 &< a_3 \pi_{00} + a_{31} \pi_{10} + a_{32} \pi_{01} + \pi_{11} \label{c33},
\end{align}
where
\begin{align*}
\pi_{00} &= \pr(Y_1=0,Y_2=0), \\
\pi_{01} &= \pr(Y_1=0,Y_2>0), \\
\pi_{10} &= \pr(Y_1>0,Y_2=0), \\
\pi_{11} &= \pr(Y_1>0,Y_2>0),
\end{align*}
and $Y=(Y_1,Y_2)$ is a random vector distributed according to the stationary number of customers
in queues~1 and~2 given that the length of queue~3 is infinite, which is well-defined when
inequalities~\reff{c31} and~\reff{c32} hold. To the best of our knowledge, there are no
closed-form expressions available for the probabilities $\pi_{00}, \pi_{01}, \pi_{10}, \pi_{11}$,
so they must be evaluated numerically.

The above result coincides with the stability characterization derived earlier by Fayolle,
Malyshev, and Menshikov using Foster--Lyapunov criteria for deflected random walks
in~$\Z_+^3$~\cite[Theorem 4.4.4]{fayolle1995}.

\subsection{Two interfering wireless base stations with channel-aware scheduling}
\label{sec:twoProcessors}

We will now turn our attention to a system of two queues served at the state-dependent rates
\[
 \phi_i(x) = g_i(x_i) h_i(x),
\]
where $g_i$ is a bounded increasing function on $\Z_+$, and $h_i$ is a decreasing function on
$\Z_+^2$. This particular form of allocation function arises as a model for two interfering
wireless base stations operating according to a channel-aware scheduling discipline, where the
functions $h_i$ model the interference between the base stations, and the functions~$g_i$
represent the scheduling gain, which increases in the number of customers due to multiuser
diversity~\cite{liu2003}. Denote
\begin{align*}
 g_i^* &= \lim_{x_i\to \infty} g_i(x_i), \\
 h_i^* &= \lim_{r \to \infty} \sup_{x_1,x_2 \ge r} h_i(x), \\
 h_i^j(x_i) &= \limsup_{x_j \to \infty} h_i(x), \quad j\neq i,
\end{align*}
and let
\[
 \bar h_j^i(\lambda_i,\phi) = \sum_{x_i} h_j^i(x_i) \, \pi^i(x_i), \quad j\neq i,
\]
where $\pi^i$ is the probability distribution ($c$ is a normalization constant) given by
\[
 \pi^i(x_i) = c \prod_{z=1}^{x_i}\frac{\lambda_i}{g_i(z) h_i^j(z)}, \quad j \neq i.
\]
The stability region can now be described using Theorem~\ref{the:boundaryDecr} as the set of
$(\lambda_1,\lambda_2)$ such that either
\begin{eqnarray*}
 \lambda_1 < g_1^* h_1^*
 &\mbox{ and }& \lambda_2 < g_2^* \bar h_2^1(\lambda_1,\phi), \\
 \mbox{or } \lambda_2 < g_2^* h_2^*
 &\mbox{ and }& \lambda_1 < g_1^* \bar h_1^2(\lambda_2,\phi).
\end{eqnarray*}

Figure~\ref{fig5} represents the full and partial stability regions of the system, where the
scheduling gains are given by
\begin{equation}
 \label{eq:gain}
 g_i(x_i) = \min(3,\log(1+x_i)),
\end{equation}
and the interference functions are of the form
\begin{equation}
 \label{eq:interference1}
 h_i(x_j) = \frac{1}{6 - 4 e^{-\gamma x_j}}, \quad j\neq i.
\end{equation}
\begin{figure}[h]
 \begin{center}
   \psfrag{S}{$\cS$}
   \psfrag{S1}{$\cS_1$}
   \psfrag{S2}{$\cS_2$}
   \psfrag{I}{$\cU$}
   \psfrag{l1}{$\lambda_1$}
   \psfrag{l2}{$\lambda_2$}
   \includegraphics[width=6cm,angle=0]{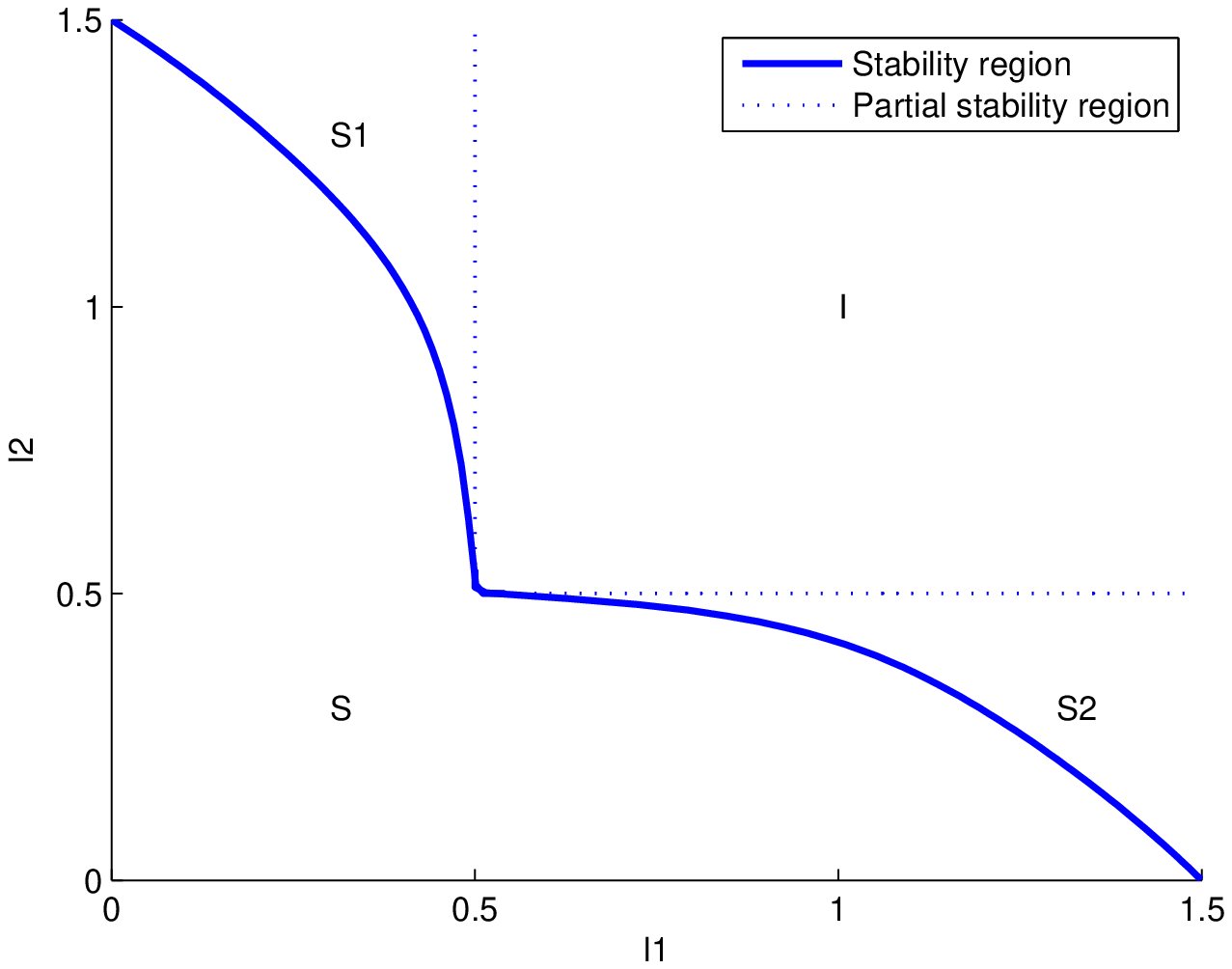}
   \includegraphics[width=6cm,angle=0]{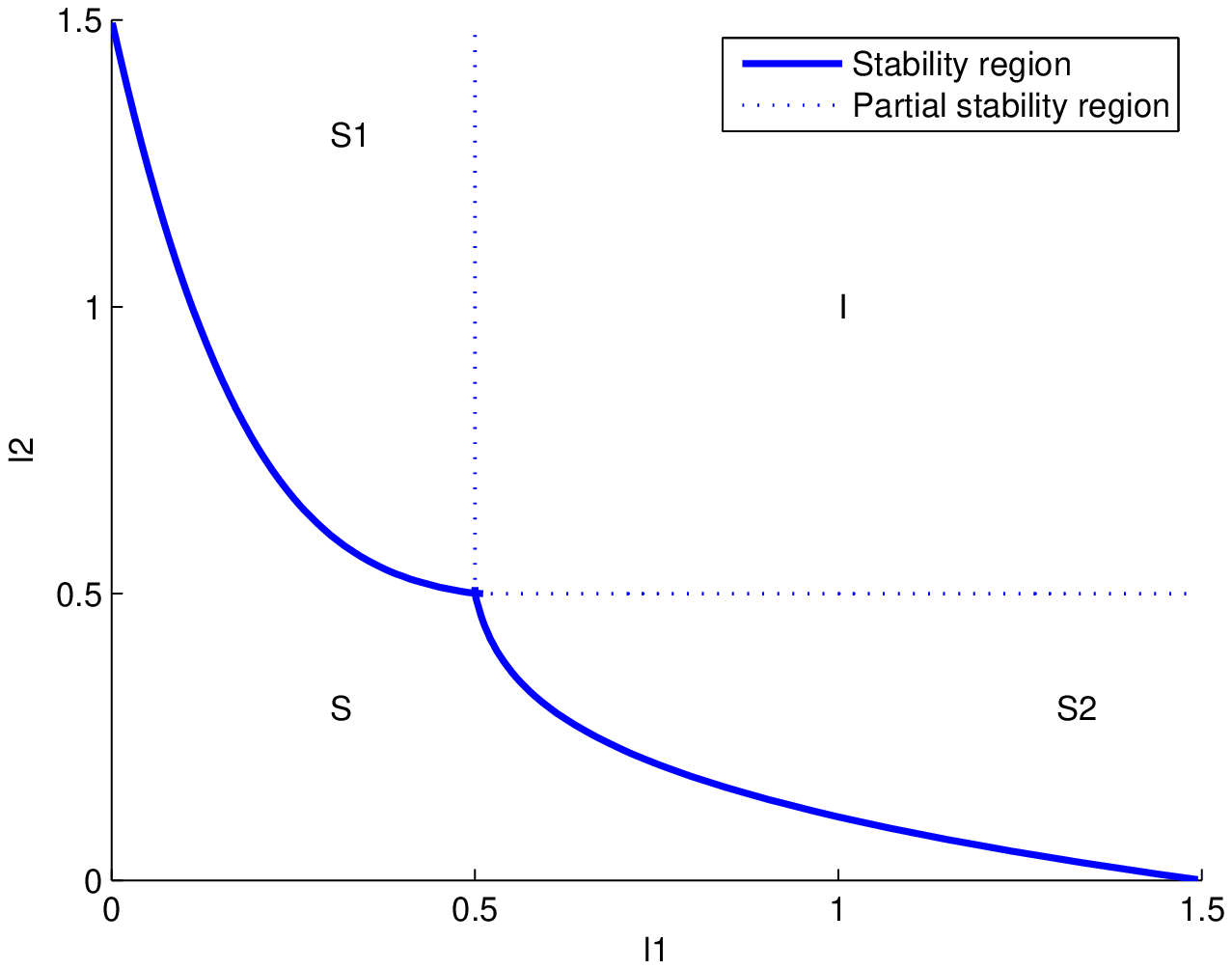}
   \caption{Stability regions for interference functions of the form~\reff{eq:interference1} with $\gamma=0.05$ (left) and $\gamma=2.0$ (right).}
   \label{fig5}
 \end{center}
\end{figure}
The area $\cS$ is the set of arrival rates such that both queues are stable, while the areas
$\cS_i$ correspond to the set of arrival rates so that only queue $i$ is stable, and $\cU$ is the
set of arrival rates where both queues are unstable. Figure~\ref{fig6} illustrates the
corresponding stability regions when $g_i$ are as in~\reff{eq:gain} and
\begin{equation}
 \label{eq:interference2}
 h_i(x_j) = \frac{1}{ 6 - 4(1+x_j)^{-\gamma} }, \quad j \neq i.
\end{equation}

\begin{figure}[h]
 \begin{center}
   \psfrag{S}{$\cS$}
   \psfrag{S1}{$\cS_1$}
   \psfrag{S2}{$\cS_2$}
   \psfrag{I}{$\cU$}
   \psfrag{l1}{$\lambda_1$}
   \psfrag{l2}{$\lambda_2$}
   \includegraphics[width=6cm,angle=0]{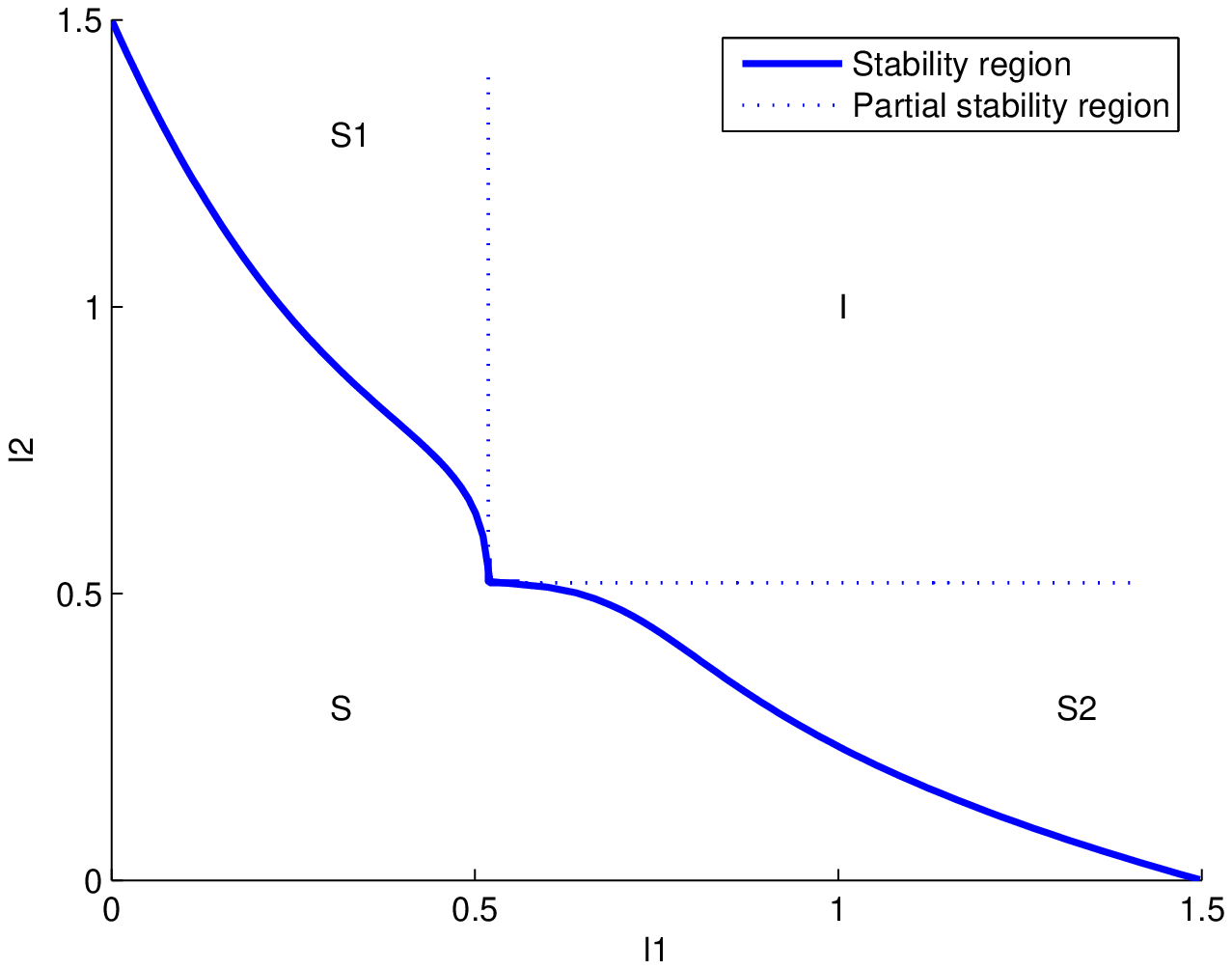}
   \includegraphics[width=6cm,angle=0]{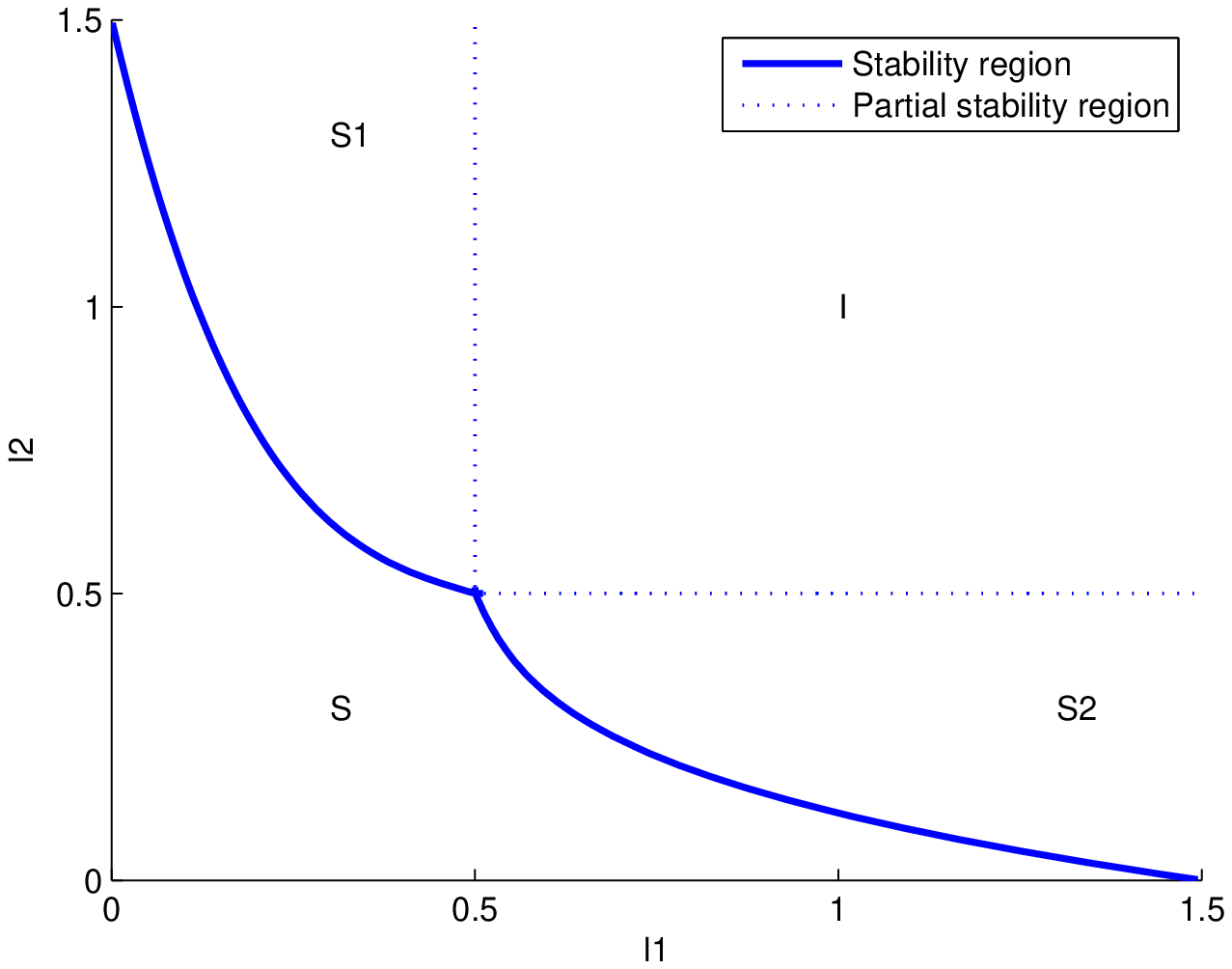}
   \caption{Stability regions for interference functions of the form~\reff{eq:interference2} with $\gamma=0.4$ (left) and $\gamma=2.0$ (right).}
   \label{fig6}
 \end{center}
\end{figure}

\section{Conclusion}
\label{sec:conclusion}
We provided sufficient and necessary conditions for the stability of a parallel queueing system
with coupled service rates, and showed that these conditions are sharp when the service rate at
each queue is decreasing in the number of customers in other queues, and has uniform limits as the
queue lengths tend to infinity. Moreover, we presented conditions for partial stability, where
only some of the queues are stable. The most general stability conditions, although not sharp, may
yield useful inner and outer bounds for the stability region of systems that are too complex to
characterize exactly. An interesting and important direction for future research is to study
whether the given results extend to the case where the service allocation does not have uniform
limits, and the service times distributions are nonexponential.

\subsection*{Acknowledgements} The work reported in this paper was carried out while
M.~Jonckheere and L.~Leskelä were employed by the Centrum Wiskunde \& Informatica, the
Netherlands. The research has been supported by the Dutch BSIK/BRICKS PDC2.1 project.

\appendix

\section{Appendix}
\label{sec:appendix}

\subsection{Small perturbations of transition rates}
\label{sec:perturbation}

\begin{lemma}
\label{the:continuity}
Let $X$ and $X^n$ be continuous-time Markov processes on a countable state space having transition
rates $q(x,y)$ and $q^n(x,y)$, and unique stationary distributions~$\pi$ and~$\pi^n$,
respectively. Assume that
\begin{enumerate}[(i)]
\item $q^n(x,y) \to q(x,y)$ as $n \to \infty$ for all~$x$ and $y$,
\item the set $\{ x: q^n(x,y) \neq 0 \ \text{for some} \ n \}$ is finite for all $y$,
\item $\{\pi^n\}_{n \ge 0}$ is a tight family of probability measures.
\end{enumerate}
Then $\pi^n(x) \to \pi(x)$ for all~$x$.
\end{lemma}
\begin{proof}
Let us assume that $\pi^n(z)$ does not converge to~$\pi(z)$ for some~$z$. Then there exists an
$\epsilon>0$ and a subsequence $\Z_+' \subseteq \Z_+$ such that $|\pi^n(z) - \pi(z)| \ge \epsilon$
for all $n \in \Z_+'$. Because $\{\pi^n\}_{n \in \Z_+'}$ is tight, there exists a further
subsequence $\Z_+'' \subseteq \Z_+'$ such that $\pi^n$ converges weakly to a probability
measure~$\tpi$ as $n \to \infty$ along~$\Z_+''$ \cite[Proposition 5.21]{kallenberg2002}. Observe
that for all~$y$ and for all~$n$,
\[
\sum_x \pi^n(x) q^n(x,y) = 0.
\]
By virtue of assumption~(ii), we can take $n \to \infty$ along~$\Z_+''$ on both sides of the above
equation, and bring the limit inside the sum, which shows that
\[
\sum_x \tpi(x) q(x,y) = 0.
\]
Because we assumed the stationary distribution of $X$ to be unique, it follows that $\tpi = \pi$,
and hence $\pi^n \to \pi$ weakly along~$\Z_+''$. This is a contradiction, because $|\pi^n(z) -
\pi(z)| \ge \epsilon$ for all $n \in \Z_+'$.
\end{proof}

\subsection{Stable multiclass birth and death processes with
strictly positive birth rates}
\label{sec:stationaryMeasure}

\begin{lemma}
\label{the:increasingSets}
Let $X$ be a $N$-class birth and death process with birth rates~$\lambda_i$ and bounded death
rates $\phi_i(x)$, and assume that $\lambda_i>0$ for all~$i$. Then the hitting time of~$X$ into an
arbitrary increasing set~$A$ is almost surely finite, regardless of the initial state.
\end{lemma}

\begin{proof}
Let $A$ be an increasing set. Then $re \in A$ for some positive integer~$r$, where
$e=(1,\dots,1)\in \Z_+^N$. Let $\hX$ be the discrete-time jump chain of~$X$, with $\hX(n)$ being
the value of~$X$ at its $n$-th jump. Then for all~$x$,
\[
\pr_x(\hX(1) = x+e_i) = \frac{\lambda_i}{\sum_j \lambda_j + \sum_{j:x_j>0} \phi_j(x)},
\]
and because $\hX$ can reach~$A$ from any state~$x$ by taking $r$~upward jumps into all coordinate
directions, we see that $\pr_x(\hX(r^N) \in A) \ge \delta$, where
\[
\delta = \left(\frac{ \min_j \lambda_j }{ \sum_j \lambda_j + N ||\phi||} \right)^{rN} > 0.
\]
By induction, it then follows that for all~$x$ and all~$M$,
\[
\pr_x(\hX(mr^N) \notin A \ \forall m=1,\dots,M) \le (1-\delta)^M.
\]
Thus, by taking $M \to \infty$, we may conclude that $\pr_x(\htau_A < \infty) = 1$, which is
equivalent to $\pr_x(\tau_A < \infty) = 1$.
\end{proof}

\begin{proof}[Proof of Proposition~\ref{the:stationaryDistribution}]

We prove that $(i) \Rightarrow (ii) \Rightarrow (iii)$, the reverse direction being clear. Let us
denote $x \to y$ if the process~$X$ started in~$x$ can reach~$y$, and let $C(x) = \{y: x \to y \
\text{and} \ y \to x\}$ be the communicating class associated with~$x$. Recall that a set~$C$ is
said to be absorbing if $x \to y$ implies $y \in C$ for all $x \in C$. Observe first that because
all birth rates of~$X$ are strictly positive, it follows that for all~$x$ and~$y$,
\begin{equation}
\label{eq:increasingReaching}
x \le y \implies x \to y.
\end{equation}
From~\reff{eq:increasingReaching} we see that all absorbing sets are increasing. Moreover, if a
communicating class~$C$ is increasing, and if $x \to y$ for some $x \in C$, then there exists
a~$z$ such that $x \le z$ and $y \le z$. Hence, $y\to z$ by~\reff{eq:increasingReaching} and $z
\in C$, because $C$ is increasing. Because $C$ is a communicating class, it follows that $y \in
C$. We may thus conclude that any communicating class~$C$ is absorbing if and only if it is
increasing.

We next show that $X$ has a unique absorbing communicating class. Assume first that all
communicating classes are nonabsorbing. Then none of the communicating classes $C(x)$ is
increasing, and Lemma~\ref{the:increasingSets} implies that $\pr_x(\tau_{C(x)^c} < \infty)$ for
all~$x$. Because $\pr_y(\tau_{C(x)} = \infty) = 1$ for all $y \notin C$, it follows that with
probability one, $X$ eventually leaves any finite set without ever returning, regardless of the
initial state. Thus, $X(t) \to \infty$ almost surely, which contradicts the assumption that $X$
started in some initial state, say $x^0$, is stable. Hence, $X$ must have at least one absorbing
communicating class. To see that there is no more than one such class, it suffices to observe that
if $C(x)$ and $C(y)$ are disjoint sets, then they can not both be increasing.

Now let $C$ be the unique absorbing communicating class of~$X$, and assume that $X[x]$ is unstable
for all $x \in C$. Then for any finite set~$K$ and any $0 \le s \le t$ it follows that
\begin{equation}
\label{eq:tauC}
\pr_{x^0}(X(t) \in K, \tau_C \le s) = \sum_{x \in C} \pr_{x^0}(X(s)=x, \tau_C \le s) \pr_x(X(t-s)
\in K),
\end{equation}
because $X(s)$ belongs to the absorbing set~$C$ on the event $\{\tau_C \le s\}$. Hence, by
dominated convergence, $\lim_{t \to \infty} \pr_{x^0}(X(t) \in K, \tau_C \le s) = 0$ for all $s$.
Furthermore, because
\[
\pr_{x^0}(X(t) \in K) \le \pr_{x^0}(\tau_C > s) + \pr_{x^0}(X(t) \in K, \tau_C \le s),
 \]
and because $\tau_C$ is finite almost surely, we see by taking first $t \to \infty$ and then $s
\to \infty$ that $\pr_{x_0}(X(t) \in K) \to 0$, which contradicts the stability of $X[x^0]$.
Hence, we may conclude that $X[y]$ is stable for some $y \in C$. Moreover, because $\pr_x(X(1) =
y) > 0$ for all~$x$, we see that for all finite sets~$K$,
\[
\limsup_{t \to \infty} \pr_x(X(t+1) \in K) \ge \limsup_{t \to \infty} \pr_x(X(1) = y) \pr_y(X(t)
\in K) > 0,
\]
so $X[x]$ is stable for all initial states $x \in \Z_+^N$.

Finally, let $X^C$ be the Markov process on the state space~$C$ with the same transition rates
as~$X$ in~$C$. Then $X^C$ is irreducible and stable, regardless of the initial state. Hence, it
follows \cite[Theorem 12.25]{kallenberg2002} that $X^C$ is positive recurrent, and thus has a
unique stationary distribution~$\pi^C$ on~$C$ such that the distribution of $X^C(t)$ converges
to~$\pi^C$ in total variation. By defining $\pi(B) = \pi^C(B\cap C)$, it follows that $\pi$ is
stationary for the unrestricted version of~$X$, and because $\pr_x(\tau_C < \infty)$ for all~$x$,
one can verify using~\reff{eq:tauC} that the distribution of $X(t)$ converges to~$\pi$ in total
variation, regardless of the initial state.

Having proved the equivalence of (i)--(iii), let now assume that $X[x]$ is unstable for all~$x$
and show that $X(t) \to \infty$ in probability regardless of the initial state. We saw above that
if all communicating classes of~$X$ are nonabsorbing, then $X(t)$ is transient, so let us assume
that $X$ has the unique absorbing class~$C$. Then $X[x]$ is irreducible and positive recurrent for
all $x \in C$, so it follows from standard theory \cite[Theorem 12.25]{kallenberg2002} that for
any finite set~$A$, the function $h_A(x,t) = \pr_x(X(t) \in A)$ tends to zero as $t \to \infty$
for all $x\in C$. Denoting the hitting time of~$X$ into~$C$ by~$\tau_C$, the strong Markov
property implies that for all~$x$,
\begin{align*}
\pr_x(X(t) \in A)
&= \pr( X(t) \in A, \tau_C \le t) + \pr( X(t) \in A, \tau_C > t) \\
&= \E_x h_A(X(\tau_C), t-\tau_C) 1(\tau_C\le t) + \pr_x( X(t) \in A, \tau_C > t).
\end{align*}
Because $\pr_x(\tau_C < \infty) = 1$, it then follows from dominated convergence that the
right-hand side in the above equality converges to zero as $t \to \infty$. Thus, $X(t) \to \infty$
in probability.

To see that $X(t) \to \infty$ implies the instability of~$X$, let us assume that $X$ is stable.
Then by choosing a finite set~$A$ such that $\pi(A)>0$, we see that $\pr_x(X(t) \in A) > 0$ for
large~$t$. This contradicts the fact that $X(t) \to \infty$ in probability, so $X$ must be
unstable.
\end{proof}

\subsection{Uniform limits of monotone functions}
\label{sec:uniformLimits}

\begin{proof}[Proof of Proposition~\ref{the:uniformLimits}]
(i) We show that when $f$ is decreasing in all its input variables, the uniform limits of~$f$ are
given by
\begin{align*}
f^0 &= \inf_x f(x), \\
f^{n,\sigma}(x_{\sigma(1)},\dots,x_{\sigma(n)}) &= \inf_{x_{\sigma(n+1)},\dots,x_{\sigma(N)}}
f(x_1,\dots,x_N).
\end{align*}
Observe first that given $\epsilon>0$, there exists~$y$ such that $|f(y)-f^0| \le \epsilon$.
Hence, by defining $r = \max(y_1,\dots,y_N)$, it follows from the monotonicity of~$f$ that
\begin{equation}
\label{eq:limDecr0}
\sup_{x: x_1,\dots,x_N > r} |f(x)-f^0| \le \epsilon.
\end{equation}
This shows that the assertion holds for $N=1$.

To proceed by induction, let us assume that the claim holds for all positive decreasing functions
on $\Z_+^{N-1}$. Let $f$ be a positive and decreasing function on $\Z_+^N$, let $\epsilon>0$, and
choose a permutation~$\sigma$. By symmetry, we assume without loss of generality that $\sigma$ is
the identity permutation, and denote $f^n = f^{n,\sigma}$. Using~\reff{eq:limDecr0}, we can first
choose an $r_0$ such that $|f(x)-f^0| \le \epsilon/2$ when $x_1,\dots,x_N > r_0$. Then by the
monotonicity of $f$, it follows that $|f^n(x_1,\dots,x_n) - f^0| \le \epsilon/2$ for all
$x_1,\dots,x_n > r_0$. Thus,
\begin{equation}
\label{eq:limDecr1}
|f(x) - f^n(x_1,\dots,x_n)| \le \epsilon
\end{equation}
for all~$x$ such that $x_1,\dots,x_N > r_0$.

Let us next choose an $i \in \{1,\dots,n\}$ and $y_i \in \{0,\dots,r_0\}$. Then the function
\[
(x_1,\dots,x_{i-1},x_{i+1},\dots,x_N) \mapsto f(x_1,\dots,x_{i-1},y_i,x_{i+1},\dots,x_N)
\]
is decreasing on $\Z_+^{N-1}$. Hence, by the induction assumption we can choose a number
$r_i(y_i)$ such that
\[
|f(x) - \inf_{x_{n+1},\dots,x_N} f(x)| \le \epsilon
\]
for all~$x$ such that $x_i=y_i$ and $x_{n+1},\dots,x_N>r(y_i)$. In particular, by defining $r_i =
\max(r_i(0),\dots,r_i(r_0))$, it follows that
\begin{equation}
\label{eq:limDecr2}
|f(x) - f^n(x_1,\dots,x_n) | \le \epsilon
\end{equation}
for all~$x$ such that $x_i \le r_0$ and $x_{n+1},\dots,x_N > r_i$. Finally, by defining $r =
\max(r_0,r_1,\dots,r_n)$, we see by combining~\reff{eq:limDecr1} and~\reff{eq:limDecr2} that
\[
\sup_{x \in \Z_+^N: x_{n+1},\dots,x_N > r} |f(x) - f^n(x_1,\dots,x_n)| \le \epsilon,
\]
which completes the induction step.

(ii) To see that $fg$ has uniform limits at infinity, it suffices to note that for any~$n$ and
any~$\sigma$ (omitting the arguments of the functions),
\[
| fg - f^{n,\sigma} g^{n,\sigma} | \le || f || | g - g^{n,\sigma} | + || g || | f - f^{n,\sigma}
|,
\]
and the same obviously holds for~$f^0$ and~$g^0$ in place of $f^{n,\sigma}$ and $g^{n,\sigma}$.
Hence, the claim for $fg$ follows by taking the supremum over $x \in \Z_+^N$ such that
$x_{\sigma(n+1)},\dots,x_{\sigma(N)} > r$ on both sides of the above inequality, and then
letting~$r$ tend to infinity. The proof for $f+g$ is analogous.
\end{proof}

\subsection{Strong law of large numbers for a time-changed Poisson process}
\label{sec:LLN}

\begin{lemma}
\label{the:LLN}
Let $N$ be a unit-rate Poisson process, and let $A$ be a stochastic process with paths in
$D(\R_+,\R_+)$ such that
\begin{equation}
 \label{eq:LLNAssumption}
 \lim_{t \to \infty} t^{-1} \int_0^t A(s) \, ds = c
  \quad \text{a.s.}
\end{equation}
for some constant $c \ge 0$. Then
\begin{equation}
  \label{eq:LLN}
  \lim_{t \to \infty} t^{-1} N\left( \int_0^t A(s) \, ds \right) = c
  \quad \text{a.s.}
\end{equation}
\end{lemma}
\begin{proof}
We will view $N$ and $A$ as $D(\R_+,\R_+)$-valued measurable mappings on a probability space
$(\Omega,\cF,\pr)$. Fix a set $\Omega_1 \subset \Omega$ with $\pr(\Omega_1)=1$ such
that~\reff{eq:LLNAssumption} holds for all realizations $A = A_\omega$ with $\omega \in \Omega_1$.
The law of large numbers for the Poisson process (e.g.\ Asmussen~\cite[Proposition
V.1.4]{asmussen2003}) shows that there exists $\Omega_2 \subset \Omega$ with $\pr(\Omega_2)=1$
such that $t^{-1} N_\omega(t) \to 1$ for all $\omega \in \Omega_2$. As a consequence,
\begin{equation}
 \label{eq:LLNPoisson}
 t^{-1} N_\omega(\lambda t) \to \lambda
\end{equation}
for all $\omega \in \Omega_2$ and $\lambda \ge 0$.

Fix an $\omega \in \Omega_1 \cap \Omega_2$. Given an arbitrary $\epsilon > 0$, we may then choose
a $t_0 = t_0(\omega,\epsilon)$ such that
\[
 (c-\epsilon)_+ t \le \int_0^t A_\omega(s) \, ds \le (c+\epsilon)t
\]
for all $t > t_0$, where we denote $u_+ = \max(u,0)$. The above inequalities imply that
\[
 t^{-1} N_\omega((c-\epsilon)_+ t) \le t^{-1} N_\omega \left( \int_0^t A_\omega(s) \, ds \right) \le t^{-1} N_\omega((c+\epsilon)t)
\]
for all $t > t_0$. With the help of~\reff{eq:LLNPoisson}, the claim follows by letting first $t
\to \infty$ and then $\epsilon \to 0$.
\end{proof}

\bibliographystyle{apalike}
\bibliography{lslReferences}

\end{document}